\documentclass[10pt]{article}
\usepackage{amssymb}
\usepackage{amsmath}
\usepackage{amsthm}
\usepackage{amsfonts}

\usepackage{a4wide}
\usepackage{longtable}
\usepackage{multirow}

\usepackage[latin1]{inputenc}
\usepackage[english]{babel}
\usepackage[T1]{fontenc}
\usepackage{amssymb,amsmath,amsthm,amstext,amsfonts,latexsym}

\usepackage{pgf,tikz}
\usepackage{mathrsfs}
\usetikzlibrary{arrows}

\newtheorem{lemma}{Lemma}[section]
\newtheorem{proposition}{Proposition}[section]
\newtheorem{theorem}{Theorem}[section]
\newtheorem{corollary}{Corollary}[theorem]
\newtheorem{conjecture}{Conjecture}[section]

\theoremstyle{definition}

\theoremstyle{remark}
\newtheorem{remark}{Remark}[section]

\usepackage[all]{xy}
\input xy
\xyoption {all}
\usepackage[pdftex,             %
    bookmarks=true,             
    hypertexnames=false,        
    bookmarksnumbered=true,     
    pdfpagemode=None,           
    pdfstartview=FitH,          
    pdfpagelayout=SinglePage,   
    colorlinks=true,            
    urlcolor=blue,              
    citecolor=blue,             
    pdfborder={0 0 0}           
    ]{hyperref}
    
\newcommand{\rank}{\operatorname{rank}\,}
\date{ }

\begin{document}

\title{ On a conjecture of Lemmermeyer }

\author{S. AOUISSI, M. TALBI,  M. C. ISMAILI and A. AZIZI}
\maketitle

\medskip
\noindent
\textbf{Abstract:}
Let $p\equiv 1\,(\mathrm{mod}\,3)$ be a prime
and denote by $\zeta_3$ a primitive third root of unity.
Recently, Lemmermeyer presented a conjecture about $3$-class groups
of pure cubic fields $L=\mathbb{Q}(\sqrt[3]{p})$
and of their normal closures $\mathrm{k}=\mathbb{Q}(\sqrt[3]{p},\zeta_3)$.
The purpose of this paper is to prove Lemmermeyer's conjecture.

\bigskip
\noindent{\bf Keywords:} {Pure cubic fields, Galois closure, $3$-class groups, abelian type invariants. \\ }
\bigskip
\noindent{\bf Mathematics Subject Classification 2010:} {11R11, 11R16, 11R20, 11R27, 11R29, 11R37.}
\bigskip

\section{Introduction}
\label{sec:Introd}

Let $L=\mathbb{Q}(\sqrt[3]{d})$ be a pure cubic field,
where $d>1$ is a cubefree positive integer,
$\mathrm{k}=\mathbb{Q}(\sqrt[3]{d},\zeta_3)$ be its normal closure,
and $C_{\mathrm{k},3}$ be the $3$-component of the class group of $\mathrm{k}$. \\
In a collection of unsolved problems,
Lemmermeyer proposed a conjecture for the special pure cubic field
$\mathbb{Q}(\sqrt[3]{p})$, where $p\equiv 1\,(\mathrm{mod}\,3)$ is a prime number
\cite[Conjecture 5, \S\ 1.10, p. 44]{FLEM}.
This conjecture gives a necessary and sufficient condition for the $3$-class group $C_{\mathrm{k},3}$
to be isomorphic to either $\mathbb{Z}/3\mathbb{Z}$ or $\mathbb{Z}/3\mathbb{Z}\times\mathbb{Z}/3\mathbb{Z}$, 
and thus specifies the $\rank$ of $C_{\mathrm{k},3}$ as follows:


\begin{conjecture}
\label{Conj1}
Let $L=\mathbb{Q}(\sqrt[3]{p})$ be a pure cubic field,
where $p$ is a prime number such that $p\equiv 1\,(\mathrm{mod}\,3)$,
and $\mathrm{k}=\mathbb{Q}(\sqrt[3]{p},\zeta_3)$ be its normal closure.
Let $C_{L,3}$ (resp. $C_{\mathrm{k},3}$) be the $3$-component of the class group of $L$ (resp. $\mathrm{k}$).
Then:
\begin{itemize}
\item[1)] $C_{L,3}$ is a cyclic $3$-group,
and if it contains a cyclic subgroup of order $9$, then $p\equiv 1\,(\mathrm{mod}\,9)$.
\item[2)] If $p\equiv 4,7\,(\mathrm{mod}\,9)$, then:
$$
C_{\mathrm{k},3}\simeq\left\{
   \begin{array}{@{} l c @{}}
      \mathbb{Z}/3\mathbb{Z}     & \text{ if~~} \left(\frac{3}{p}\right)_3\neq 1, \\
      (\mathbb{Z}/3\mathbb{Z})^2 & \text{ if~~} \left(\frac{3}{p}\right)_3=1.
   \end{array}\right.
$$
\item [3)] If $p\equiv 1\,(\mathrm{mod}\,9)$,
then $\rank C_{\mathrm{k},3}\in\{1,2\}$, independently of the value of $\left(\frac{3}{p}\right)_3$.

\end{itemize}
Here $\left(\frac{.}{p} \right)_3$ is the cubic residue symbol.
\end{conjecture}

In fact, Conjecture \ref{Conj1} for $p\equiv 4,7\,(\mathrm{mod}\,9)$ was first expressed in 1970 by Barrucand and Cohn
\cite[\S\ 8, p. 19]{BC2},
partially proved in 1976 by Barrucand, H. C. Williams and Baniuk
\cite[\S\ 7, Thm. 1, p. 321, and \S\ 8, Cnj. 1, p. 322]{BWB},
and mentioned again in 1982 by H. C. Williams
\cite[\S\ 6, p. 273]{HCW1982}.
Conjecture \ref{Conj1} for $p\equiv 1\,(\mathrm{mod}\,9)$
was proved partially in 2005 by Gerth
\cite[Formulas p. 474, and Case 4, pp. 475--476]{Ge2005},
who also pointed out that Conjecture \ref{Conj1} for $p\equiv 4,7\,(\mathrm{mod}\,9)$ is still an open problem.

Based on results concerning the $3$-class group $C_{\mathrm{k},3}$ in \S\ \ref{ss:Preliminaries},
we shall prove Conjecture \ref{Conj1} in \S\ \ref{ss:Conj1}.
It will be underpinned by numerical examples
obtained with the Computational Number Theory System PARI \cite{PARI} in \S\ \ref{NumExp}. Throughout this paper, we will use the following notations: 

\begin{itemize}
 \item  $p$ is a prime number such that $p\equiv 1 \pmod 3$;
 \item $L= \mathbb{Q}(\sqrt[3]{d})$ is a pure cubic field, where $d>1$ is a cube-free positive integer;
 \item $\mathrm{k}_0= \mathbb{Q}(\zeta_3)$, where $\zeta_3=e^{2i\pi/3}$ denotes a primitive third root of unity;
 \item $\mathrm{k}=\mathbb{Q}(\sqrt[3]{d},\zeta_3)$ is the normal closure of $L$;
 \item $\langle\tau\rangle=\operatorname{Gal}(\mathrm{k}/L)$ such that $\tau^2=id$, $\tau(\zeta_3)=\zeta_3^2$ and $\tau(\sqrt[3]{d})=\sqrt[3]{d}$;
 \item $\langle\sigma\rangle=\operatorname{Gal}(\mathrm{k}/\mathrm{k}_0)$ such that $\sigma^3=id$, $\sigma(\zeta_3)=\zeta_3$,
   $\sigma(\sqrt[3]{d})=\zeta_3\sqrt[3]{d}$ and $\tau\sigma=\sigma^2\tau$;
 \item $\lambda=1-\zeta_3$ and $\pi$ are prime elements of $\mathrm{k}_0$;
 \item $q^{\ast}=1$ or \(0\) according to whether \(\zeta_{3}\) is or is not norm of an element of \(\mathrm{k}\setminus\lbrace 0\rbrace\);
 \item $u$ denotes the index of the subgroup $E_0$ generated by the units of intermediate fields of the extension $\mathrm{k}/\mathbb{Q}$ in the group of units of $\mathrm{k}$;
 \item $\mathcal{N}_{\mathrm{k}/\mathrm{k}_{0}}$ denotes the norm of $\mathrm{k}$ on $\mathrm{k}_{0}$;
 \item \(t\) denotes the number of prime ideals ramified in \(\mathrm{k}/\mathrm{k}_{0}\);
 \item $\left(\frac{.}{p}\right)_3$ is the cubic residue symbol such that
$\left(\frac{c}{p}\right)_3=1$ $\Leftrightarrow$
$X^3\equiv c\,(\mathrm{mod}\,p)$ has a solution in $\mathbb{Z}$ $\Leftrightarrow$
$c^{(p-1)/3}\equiv 1\,(\mathrm{mod}\,p)$, where $c\in\mathbb{Z}$,
and $p$ is a prime number such that $p\nmid c$ and $p\equiv 1\,(\mathrm{mod}\,3)$;
 \item For an algebraic number field $F$:
  \begin{itemize}
  \item $\mathcal{O}_{F}$, $E_{F}$ : the ring of integers and the group of units of $F$;
  \item $C_{F,3}$, $\mathrm{F}_3^{(1)}$ : the $3$-class group and the Hilbert $3$-class field of $F$;
  \item $[\mathcal{I}]$ : the class of a fractional ideal $\mathcal{I}$ in the class group of $F$;
  \end{itemize}
\end{itemize}




\section{Proof of Conjecture \ref{Conj1}}
\label{sec:Proof}

\subsection{Preliminary results}
\label{ss:Preliminaries}

Let $d>1$ be a cubefree integer,
and let $L=\mathbb{Q}(\sqrt[3]{d})$ be the pure cubic field with radicand $d$.
We denote by $\zeta_3=(-1+i\sqrt{3})/2$ a primitive cube root of unity.
By $\mathrm{k}_0$ we denote the third cyclotomic field $\mathbb{Q}(\zeta_3)$,
by $\mathrm{k}=\mathbb{Q}(\sqrt[3]{d},\zeta_3)$ the normal closure of the pure cubic field $L$,
and by $C_{\mathrm{k},3}$ the $3$-component of the class group of $\mathrm{k}$.
Further, let $\langle\tau\rangle=\operatorname{Gal}\left(\mathrm{k}/L\right)$
and $\langle\sigma\rangle=\operatorname{Gal}\left(\mathrm{k}/\mathrm{k}_0\right)$.


The $3$-class group $C_{\mathrm{k},3}$ can be viewed as a $\mathbb{Z}_3[\zeta_3]$-module.
According to  \cite[\S\ 2, Lemma 2.1 and Lemma 2.2, p. 53]{GERTH2} we have:
$$ C_{\mathrm{k},3} \cong C_{\mathrm{k},3}^{+} \times C_{\mathrm{k},3}^{-} \ \ \text{and} \  \ C_{\mathrm{k},3}^{+}\cong C_{L,3}.$$
Define the $3$-group $C_{\mathrm{k},3}^{(1-\sigma)^{i}}$ for each $i\in\mathbb{N}$ by
$C_{\mathrm{k},3}^{(1-\sigma)^{i}}=
\lbrace \mathcal{A}^{(1-\sigma)^{i}}~\mid~ \mathcal{A} \in C_{\mathrm{k},3} \rbrace$.
Since we always have $(1-\zeta_3)^2\cdot\mathbb{Z}_3[\zeta_3]=3\cdot \mathbb{Z}_3[\zeta_3]$,
then for each $i \in \mathbb{N}$ we have
$C_{\mathrm{k},3}^{(1-\sigma)^{i+2}}=(C_{\mathrm{k},3}^{(1-\sigma)^{i}})^3$.
Consequently, we have the following equation for the $\rank$ of the group $C_{\mathrm{k},3}$:

 \begin{equation}
 \label{eq1}
 \rank C_{\mathrm{k},3}=\rank (C_{\mathrm{k},3}/C_{\mathrm{k},3}^{3})=
 \rank (C_{\mathrm{k},3}/C_{\mathrm{k},3}^{(1-\sigma)}) +
 \rank (C_{\mathrm{k},3}^{(1-\sigma)}/C_{\mathrm{k},3}^{(1-\sigma)^{2}}).
 \end{equation}

The fact that, for each $i\in\mathbb{N}$,
$C_{\mathrm{k},3}^{(1-\sigma)^{i}}/C_{\mathrm{k},3}^{(1-\sigma)^{i+1}}$
is a $\mathbb{Z}_3 [\langle\tau\rangle]$-module implies that:
$$\rank(C_{\mathrm{k},3}^{(1-\sigma)^{i}}/C_{\mathrm{k},3}^{(1-\sigma)^{i+1}})=
\rank (C_{\mathrm{k},3}^{(1-\sigma)^{i}}/C_{\mathrm{k},3}^{(1-\sigma)^{i+1}})^{+}
+
\rank (C_{\mathrm{k},3}^{(1-\sigma)^{i}}/C_{\mathrm{k},3}^{(1-\sigma)^{i+1}})^{-}.$$

So, for each $i\in \mathbb{N}$, we consider the homomorphism  $\varphi_{i}$  as follows:
\begin{eqnarray*}
\varphi_{i}: C_{\mathrm{k},3}^{(1-\sigma)^{i}}/C_{\mathrm{k},3}^{(1-\sigma)^{i+1}} & \longrightarrow & C_{\mathrm{k},3}^{(1-\sigma)^{i+1}}/C_{\mathrm{k},3}^{(1-\sigma)^{i+2}} \\
\mathcal{A} \ \bmod \ C_{\mathrm{k},3}^{(1-\sigma)^{i+1}} & \longmapsto & \mathcal{A}^{1-\sigma} \ \bmod \ C_{\mathrm{k},3}^{(1-\sigma)^{i+2}}
\end{eqnarray*}
where $\mathcal{A} \in C_{\mathrm{k},3}^{(1-\sigma)^{i}} $.
Since $C_{\mathrm{k},3}^{(1-\sigma)^{i}}/C_{\mathrm{k},3}^{(1-\sigma)^{i+1}}$
is an elementary abelian $3$-group for each $i\in \mathbb{N}$,
it can be viewed as vector space over $\mathbb{Z}_3$.
Thus $\varphi_{i}$ is a surjective vector space homomorphism.

Let  $\mathcal{B} \in C_{\mathrm{k},3}^{(1-\sigma)^{i}}/C_{\mathrm{k},3}^{(1-\sigma)^{i+1}}$,
then:
$$(\mathcal{B}^{1-\sigma})^{\tau}= (\mathcal{B}^{\tau})^{1-\sigma^2} = \mathcal{B}^{1-\sigma^2} =
\mathcal{B}^{3-(1-\sigma)-(1+\sigma+\sigma^2)} \equiv (\mathcal{B}^{1-\sigma})^{-1} \ \bmod \ C_{\mathrm{k},3}^{(1-\sigma)^{i+2}}$$
if $\mathcal{B} \in (C_{\mathrm{k},3}^{(1-\sigma)^{i}}/C_{\mathrm{k},3}^{(1-\sigma)^{i+1}})^{+}$,
and
$$(\mathcal{B}^{1-\sigma})^{\tau}= (\mathcal{B}^{\tau})^{1-\sigma^2} = (\mathcal{B}^{-1})^{1-\sigma^2} =
(\mathcal{B}^{-1})^{3-(1-\sigma)-(1+\sigma+\sigma^2)} \equiv \mathcal{B}^{1-\sigma} \ \bmod \ C_{\mathrm{k},3}^{(1-\sigma)^{i+2}}$$
if $\mathcal{B} \in (C_{\mathrm{k},3}^{(1-\sigma)^{i}}/C_{\mathrm{k},3}^{(1-\sigma)^{i+1}})^{-}$.
Thus for each $i$, the homomorphism  $\varphi_{i}$ maps $(C_{\mathrm{k},3}^{(1-\sigma)^{i}}/C_{\mathrm{k},3}^{(1-\sigma)^{i+1}})^{+}$
onto $(C_{\mathrm{k},3}^{(1-\sigma)^{i+1}}/C_{\mathrm{k},3}^{(1-\sigma)^{i+2}})^{-}$
and maps $(C_{\mathrm{k},3}^{(1-\sigma)^{i}}/C_{\mathrm{k},3}^{(1-\sigma)^{i+1}})^{-}$
onto $(C_{\mathrm{k},3}^{(1-\sigma)^{i+1}}/C_{\mathrm{k},3}^{(1-\sigma)^{i+2}})^{+}$.

Put $q^{*}=0$ or $1$ according to whether $\zeta_{3}$ is not or is norm of an element of $\mathrm{k}\backslash\{0\}$.
Let $t$ be the number of primes ramified in $\mathrm{k}/\mathrm{k}_{0}$
and $C_{\mathrm{k},3}^{(\sigma)} = \lbrace \mathcal{A} \in C_{\mathrm{k},3} / ~\mathcal{A}^{\sigma} = \mathcal{A} \rbrace$
be the ambiguous ideal class group of $\mathrm{k}/\mathrm{k}_{0} $,
where $\sigma$ is a generator of $\operatorname{Gal}\left(\mathrm{k}/\mathrm{k}_{0} \right)$.
Then, according to \cite[\S\ 5, pp 91-92]{GERTH1} we have
\begin{equation}
\label{eq2}
 |C_{\mathrm{k},3}^{(\sigma)}| = 3^{t-2+q^{*}}.
\end{equation}

If we denote by $C_{\mathrm{k}_0,3}$ the Sylow $3$-subgroup of the ideal class group of $\mathrm{k}_0$,
then $C_{\mathrm{k}_0,3}=\lbrace 1 \rbrace$,
and by the exact sequence :
$$1 \longrightarrow C_{\mathrm{k},3}^{(\sigma)}\longrightarrow C_{\mathrm{k},3}\overset{1-\sigma}{\longrightarrow} C_{\mathrm{k},3}
\longrightarrow C_{\mathrm{k},3}/C_{\mathrm{k},3}^{1-\sigma}\longrightarrow 1$$
we deduce that $$|C_{\mathrm{k},3}^{(\sigma)}|=| C_{\mathrm{k},3}/C_{\mathrm{k},3}^{1-\sigma}|.$$
The fact that $C_{\mathrm{k},3}^{(\sigma)}$ and $C_{\mathrm{k},3}/C_{\mathrm{k},3}^{1-\sigma}$ are elementary abelian $3$-groups implies that:
$$\rank C_{\mathrm{k},3}^{(\sigma)}= \rank( C_{\mathrm{k},3}/C_{\mathrm{k},3}^{1-\sigma}).$$

Next, we define the important \textit{index of subfield units} for the normal closure of pure cubic fields as follows:
Put $L'=\mathbb{Q}(\zeta_3\sqrt[3]{d})$ and $L''=\mathbb{Q}(\zeta_3^2\sqrt[3]{d})$.
Let $\mathcal{O}_{\mathrm{k}}$, $\mathcal{O}_{L},~\mathcal{O}_{L'},~\mathcal{O}_{L''}$ and $\mathcal{O}_{\mathrm{k}_0}$, respectively,
be the rings of integers of $\mathrm{k}$, $L$, $L'$, $L''$ and $\mathrm{k}_0$.
Let $E_{\mathrm{k},3}$ be the unit's group in $\mathcal{O}_{\mathrm{k},3},$
and let $E_{0}$ be the subgroup of $E_{\mathrm{k},3}$ generated by the units in the ring of integers
$\mathcal{O}_{L},~\mathcal{O}_{L'},~\mathcal{O}_{L''}~$ and $\mathcal{O}_{\mathrm{k}_0}$.
We let $u$ denote the index $[E_{\mathrm{k},3} : E_0]$.
According to \cite[\S\ 12, Theorem 12.1, p. 229]{B-C}, there are two possibilities, either $u=1$ or $u=3$.
\paragraph{}
To prove Conjecture \ref{Conj1}, we must employ the following Lemmas:

\begin{lemma}
\label{33}
Let $L= \mathbb{Q}(\sqrt[3]{d})$ be a pure cubic field, where
$d>1$ is a cubefree natural number, and $\mathrm{k}=\mathbb{Q}(\sqrt[3]{d},\zeta_3)$ be its normal closure.
Let $C_{L,3}$ (resp. $C_{\mathrm{k},3}$) be the $3$-component of the class group of $L$ (resp. $\mathrm{k}$),
$h_{L}$ the class number of $L$,
and $u$ the index of subfield units, defined as above. Then:
\begin{itemize}
\item[1)] $C_{\mathrm{k},3}\simeq\mathbb{Z}/3\mathbb{Z}\times\mathbb{Z}/3\mathbb{Z}$
$\Leftrightarrow$ ($3$ divides $h_{L}$ exactly and $u=3$).
\item[2)] $C_{L,3}\simeq C_{\mathrm{k},3}\simeq\mathbb{Z}/3\mathbb{Z}$
$\Leftrightarrow$ ($3$ divides $h_{L}$ exactly and $u=1$).
\end{itemize}
\end{lemma}

\begin{proof}
$\\ 1)$ If $C_{\mathrm{k},3} \simeq \mathbb{Z}/3\mathbb{Z}\times\mathbb{Z}/3\mathbb{Z}$,
then $|C_{\mathrm{k},3}|=9$ and by \cite[\S\ 14, Theorem 14.1, p. 232]{B-C}
we have $|C_{\mathrm{k},3}|=\dfrac{u}{3}|C_{L,3}|^2=9$.
We deduce that $u=3$ and $|C_{L,3}|=3$,
since the other value for $u$, namely $1$, is not possible because otherwise $27$ would be a square. \\
Conversely, if $3$ divides $h_{L}$ exactly and $u=3$,
then $3^2$ divides the class number $h_{\mathrm{k}}$ of $\mathrm{k}$ exactly
and $|C_{\mathrm{k},3}|=9$.
By  \cite[\S\ 2, Lemma 2.2, p. 53]{GERTH2},
the group $C_{\mathrm{k},3}^{+}$ is cyclic of order $3$.
On the other hand, by  \cite[\S\ 2, Lemma 2.1, p. 53]{GERTH2},
$C_{\mathrm{k},3}\simeq C_{\mathrm{k},3}^{+}\times C_{\mathrm{k},3}^{-}$.
Thus we have $|C_{\mathrm{k},3}|= |C_{\mathrm{k},3}^{+}|\cdot | C_{\mathrm{k},3}^{-}| = 9$,
so $|C_{\mathrm{k},3}^{-}|=3 $ and $C_{\mathrm{k},3}^{-}$ is also a cyclic group of order $3$.
As $C_{\mathrm{k},3}$ is the direct product of two cyclic subgroups of order 3,
then $C_{\mathrm{k},3}\simeq \mathbb{Z}/3\mathbb{Z}\times\mathbb{Z}/3\mathbb{Z}$.
$\\ 2)$ We have the same proof as above.
\end{proof}

\begin{lemma}\label{14s}
Let $p$ be a prime number such that $p\equiv 1~(\bmod~3)$. Let $L= \mathbb{Q}(\sqrt[3]{p})$, and $\mathrm{k}=\mathbb{Q}(\sqrt[3]{p},\zeta_3)$ be its normal closure. Then, $p=\pi_1\pi_2$,
with $\pi_1$ and $\pi_2$ are two primes of $\mathrm{k}_{0}$ such that
$\pi_2=\pi_1^{\tau}$ and $\pi_1 \equiv \pi_2 \equiv 1 ~(\bmod~3\mathcal{O}_{\mathrm{k}_0})$, where $\langle\tau\rangle=\operatorname{Gal}\left(\mathrm{k}/L\right)$
. Furthermore: 
$$ [ \left(\frac{\zeta_3}{\pi_1} \right)_3= \left(\frac{\zeta_3}{\pi_2} \right)_3 = 1 ] \Longleftrightarrow 
 \left(\frac{\zeta_3}{p} \right)_3=1,$$
 where $\left(\frac{.}{p} \right)_3$ is the cubic residue symbol. 
\end{lemma}

\begin{proof}
Let $p$ be a prime number such that $p\equiv 1~(\bmod~3)$. Then according to \cite[\S\ 9, Section 1, prop. 9.1.4, p.110]{Clas} there is two primes $\pi_1$ and $\pi_2$  of $\mathrm{k}_{0}$ such that $p=\pi_1\pi_2$,  $\pi_2=\pi_1^{\tau}$ and $\pi_1 \equiv \pi_2 \equiv 1 ~(\bmod~3\mathcal{O}_{\mathrm{k}_0})$.
\begin{itemize}
 \item[$\Leftarrow:$] It is clear that if
  $ \left(\frac{\zeta_3}{p} \right)_3=1$ then $\left(\frac{\zeta_3}{\pi_1} \right)_3= \left(\frac{\zeta_3}{\pi_2} \right)_3 = 1$. 
\item[$\Rightarrow:$] If $ \left(\frac{\zeta_3}{\pi_1} \right)_3= \left(\frac{\zeta_3}{\pi_2} \right)_3 = 1$, then the equations
$$
   \left\{
    \begin{array}{@{} l c @{}}
   \zeta_3 \equiv X^3 \ (\bmod \ \pi_{1}), \\
         \zeta_3 \equiv Y^3 \ (\bmod \ \pi_{2}),  
    \end{array}\right.
$$
 are solvable in $ \mathcal{O}_{\mathrm{k}_0},$ and we have  $ X \equiv Z (\bmod \ \pi_{1}) $ and $ Y \equiv Z (\bmod \ \pi_{2}) $ because $\mathcal{O}_{\mathrm{k}_0} / (\pi_1)$ is a field. Then $ \bar{X}=\bar{Z} \neq \bar{0}$ in
$\mathcal{O}_{\mathrm{k}_0} / (\pi_1)$, then 
$ \bar{X}^3=\bar{Z}^3$ in
$\mathcal{O}_{\mathrm{k}_0} / (\pi_1)$, thus  $ X^3 \equiv Z^3 (\bmod \ \pi_{1}) $ is solvable in $ \mathcal{O}_{\mathrm{k}_0}.$ Similarly for $Y$, we obtain   $Y^3 \equiv Z^3 (\bmod \ \pi_{2}) $ is solvable in $ \mathcal{O}_{\mathrm{k}_0}.$ Then
 $$
   \left\{
    \begin{array}{@{} l c @{}}
   \zeta_3 \equiv Z^3 \ (\bmod \ \pi_{1}), \\
         \zeta_3 \equiv Z^3 \ (\bmod \ \pi_{2}),  
    \end{array}\right.
$$
 are solvable in $ \mathcal{O}_{\mathrm{k}_0}$. So $ (\zeta_3-Z^3)$ is in $ \mathcal{O}_{\mathrm{k}_0}$ because  $\pi_1$ and $\pi_2$ are two different primes of $ \mathrm{k}_0 $. Since $(\pi_1, \pi_2)=1$,  then according to Gauss's Theorem we get  $\pi_1\pi_2 \mid (\zeta_3-Z^3)$ in $ \mathcal{O}_{\mathrm{k}_0}$, which implies that 
 $ \zeta_3 \equiv Z^3 (\bmod \ p) $ is solvable in $ \mathcal{O}_{\mathrm{k}_0}$, since $p=\pi_1\pi_2$. Thus 
 $ \left(\frac{\zeta_3}{p} \right)_3=1$.
 
\end{itemize}
\end{proof}

\begin{lemma}\label{14ss}
Let $p$ be a prime number such that $p\equiv 1~(\bmod~3)$.  If $ \left(\frac{\zeta_3}{p} \right)_3=1$, then $p\equiv 1~(\bmod~9)$, where $\left(\frac{.}{p} \right)_3$ is the cubic residue symbol.
\end{lemma}

\begin{proof}
Let $p$ be a prime number such that $p\equiv 1~(\bmod~3)$. Let us assume $\left(\frac{\zeta_3}{p} \right)_3=1$. \\
Since $p\equiv 1~(\bmod~3)$, then according to  \cite[\S\ 9, Section 1, prop. 9.1.4, p.110]{Clas}, $p=\pi_1\pi_2$ where $\pi_1$ and $\pi_2$ are two primes of $\mathrm{k}_{0}$ such that $\pi_2=\pi_1^{\tau}$.
According to \cite[ \S\ \S\ 7.3, Theorem 7.8, p. 217]{RECLAWS} we obtain
$\pi_1=a+b\zeta_3$ and $\pi_2=a+b\zeta_3^2$, with $a=3m+1$ and $b=3n$, where $(n,m)\in \mathbb{N}^2$. Then:
\begin{eqnarray*}
p & = & \pi_1\pi_2 \\
 & = & (a+b\zeta_3)(a+b\zeta_3^2) \\
 & = & a^2 + b^2 - ab \\
 & = & (3m+1)^2 + (3n)^2 - (3m+1)(3n) \\
 & = & 9m^2 + 9n^2 -9mn+6m-3n+1 \\
 & \equiv & 6m-3n+1 \pmod 9 \\
\end{eqnarray*}

According to \cite[ \S\ \S\ 7.3, Theorem 7.8, p. 217]{RECLAWS} we have 
$$\left(\frac{\zeta_3}{p} \right)_3= \zeta_3^{\frac{1-a-b}{3}}=\zeta_3^{-(m+n)},$$
 then
$$
  \left(\frac{\zeta_3}{p} \right)_3= \left\{
    \begin{array}{@{} l c @{}}
      1, & \text{ if~~}  m+n \equiv 0 \pmod 3, \\
         \zeta_3, & \text{ if~~}  m+n \equiv -1 \pmod 3, \\
       \zeta_3^2, & \text{ if~~}
       m+n \equiv -2 \pmod 3,
    \end{array}\right.
$$
since $\left(\frac{\zeta_3}{p} \right)_3=1$, then there exist $k \in \mathbb{N}$ such that $m+n=3k$. So
\begin{eqnarray*}
p & \equiv & 6m-3n+1 \pmod 9 \\
 & \equiv & 6m-3(3k-m)+1 \pmod 9 \\
 & \equiv & 9m-9k+1 \pmod 9 \\
 & \equiv & 1 \pmod 9.
\end{eqnarray*}
\end{proof}

\begin{lemma}
\label{11}
Let $\mathrm{k}=\mathbb{Q}(\sqrt[3]{p}, \zeta_3)$,
where $p$ is a prime number such that $p\equiv 1~(\bmod~3)$.
Let $C_{\mathrm{k},3}^{(\sigma)}$ be the ambiguous ideal class group of $\mathrm{k}/\mathrm{k}_0$,
where $\sigma$ is a generator of $\operatorname{Gal}\left(\mathrm{k}/\mathrm{k}_0\right)$.
Then $\vert C_{\mathrm{k},3}^{(\sigma)}\vert =3$.
\end{lemma}

\begin{proof}
Since $p\equiv 1 \ (\bmod \ 3$),
then according to \cite[\S\ 9, Section 1, prop. 9.1.4, p.110]{Clas}
we have $p=\pi_1\pi_2$,
where $\pi_1$ and $\pi_2$ are two primes of $\mathrm{k}_{0}$ such that
$\pi_2=\pi_1^{\tau}$ and $\pi_1 \equiv \pi_2 \equiv 1 ~(\bmod~3\mathcal{O}_{\mathrm{k}_0})$.
We study all cases depending on the congruence class of $p$ modulo $9$, then:
\begin{itemize}
\item
If $p \equiv 4\ \text{or} \ 7 \ (\bmod \ 9)$,
then according to \cite[\S\ 4, pp. 51-55]{DED},
the prime $3$ is ramified in the field $L$, so the prime ideal $(1-\zeta_3)$ is ramified in $\mathrm{k}/\mathrm{k}_0$.
Also $\pi_1$ and $\pi_2$ are totally ramified in $\mathrm{k}$.
So $t=3$.
As $p\equiv 4 \  or \  7  \pmod 9$, then $\pi_1\pi_2 = p \equiv 4 \  or \  7  \ (\bmod\  (1-\zeta_3)^4)$
because $9=3^2=(-\zeta_3^2(1-\zeta_3)^2)^2=\zeta_3(1-\zeta_3)^4$
(cf. \cite[\S\ 9, Section 1, prop. 9.1.4, p.110]{Clas}),
so $p=\pi_1\pi_2 \equiv 4 \  or \  7  \ (\bmod\  (1-\zeta_3)^3)$.
Thus $\pi_1$ and $\pi_2 \not \equiv 1  \ (\bmod\  (1-\zeta_3)^3)$,
and according to \cite[\S\ 5, pp. 91-92]{GERTH1} we obtain
$$\left(\frac{\zeta_3,p}{p} \right)_3 \neq 1$$
where the symbol $(\dfrac{ \ , \ }{})_3$ is the cubic Hilbert symbol.
We deduce that $\zeta_3$ is not a norm in the extension $\mathrm{k} / \mathrm{k}_{0} $,
so $q^*=0$.
By Equation (\ref{eq2}) we get $\vert C_{\mathrm{k},3}^{(\sigma)}\vert= 3$.
\item
If $p \equiv 1 \ (\bmod \ 9)$,
then the prime ideals which ramify in $\mathrm{k}/\mathrm{k}_0$ are $(\pi_{1})$ and $(\pi_{2})$, so $t=2$.
Moreover, $\pi_{1} \equiv \pi_2 \equiv 1 \ ( \bmod \ (1-\zeta_3)^3 )$.
Thus, according to \cite[\S\ 5, pp. 91-92]{GERTH1}, the cubic Hilbert symbol is:
$$\left(\frac{\zeta_3,p}{ \pi_1}\right)_{3}=\left(\frac{\zeta_3,p}{ \pi_2}\right)_{3}=1.$$
We conclude that $\zeta_3$ is a norm in the extension $\mathrm{k} / \mathrm{k}_{0} $,
that is, $q^*=1$, so according to Equation (\ref{eq2})
we obtain  $\vert C_{\mathrm{k},3}^{(\sigma)}\vert= 3$.
\end{itemize}

\end{proof}

Next, we specify the $\rank$ of $C_{\mathrm{k},3}$ as follows:

\begin{lemma}
\label{32}
Let $\mathrm{k}=\mathbb{Q}(\sqrt[3]{p}, \zeta_3)$, where $p$ is a prime number such that $p\equiv 1~(\bmod~3)$.
Let $s$ be the non-null positive integer such that
$C_{\mathrm{k},3}^{(\sigma)}\subseteq C_{\mathrm{k},3}^{(1-\sigma)^{s-1}} $
but $C_{\mathrm{k},3}^{(\sigma)} \not \subseteq C_{\mathrm{k},3}^{(1-\sigma)^{s}} $,
where $C_{\mathrm{k},3}^{(\sigma)}$ denotes the ambiguous ideal class group of $\mathrm{k}/\mathbb{Q}(\zeta_3)$
and $\sigma$ is a generator of $\operatorname{Gal}\left(\mathrm{k}/\mathbb{Q}(\zeta_3)\right)$.
Let $C_{\mathrm{k},3}$ be the $3$-component of the class group of $\mathrm{k}$,
then $\vert C_{\mathrm{k},3} \vert=3^{s}$, and $\rank C_{\mathrm{k},3}=1$ or $2$.
Furthermore
$$
  C_{\mathrm{k},3} \simeq  \left\{
    \begin{array}{@{} l c @{}}
      \left(\mathbb{Z}/3^{\frac{s}{2}}\mathbb{Z}\right)^{2}, & \text{ if~~} s \,\, \text{is even,} \\
       \mathbb{Z}/3^{\frac{s+1}{2}}\mathbb{Z} \times \mathbb{Z}/3^{\frac{s-1}{2}}\mathbb{Z}, & \text{ if~~}
       s \,\, \text{is odd,}
    \end{array}\right.
$$

\end{lemma}


\begin{proof}

We have $p\equiv 1~(\bmod~3)$, then by Lemma \ref{11},
$\vert C_{\mathrm{k},3}^{(\sigma)}\vert =3$.
Since $|C_{\mathrm{k},3}^{(\sigma)}|= | C_{\mathrm{k},3}/C_{\mathrm{k},3}^{1-\sigma} |$, then
$$ \vert C_{\mathrm{k},3}/C_{\mathrm{k},3}^{1-\sigma}\vert=
\vert
C_{\mathrm{k},3}^{1-\sigma}/C_{\mathrm{k},3}^{(1-\sigma)^{2}}\vert=
\ldots = \vert C_{\mathrm{k},3}^{(1-\sigma)^{s-1}}/
C_{\mathrm{k},3}^{(1-\sigma)^{s}} \vert=3,$$
\noindent
where $s$ is the positive integer defined above, and we have
$$\vert C_{\mathrm{k},3} \vert = \vert C_{\mathrm{k},3}/C_{\mathrm{k},3}^{1-\sigma}\vert
 \times \vert
C_{\mathrm{k},3}^{1-\sigma}/C_{\mathrm{k},3}^{(1-\sigma)^{2}}\vert
\times \ldots \times \vert C_{\mathrm{k},3}^{(1-\sigma)^{s-1}}/
C_{\mathrm{k},3}^{(1-\sigma)^{s}} \vert=3^s.$$
According to Equation (\ref{eq1}),
it is easy to see that
 \[ \rank C_{\mathrm{k},3} = \left\{
  \begin{array}{l l}
    1 & \quad \text{if \ $ s=1. $     }\\
        2 & \quad \text{if \ $ s>1. $   }\\
  \end{array} \right.\]

Next, let $\mathrm{k}_3^{(1)}$ be the maximal abelian unramified $3$-extension of $\mathrm{k}$.
So $\mathrm{k}_3^{(1)}/\mathrm{k}_0$ is Galois, and according to class field theory we have

\begin{equation}\label{IsomCk}
\operatorname{Gal}\left(\mathrm{k}_3^{(1)}/\mathrm{k} \right) \cong  C_{\mathrm{k},3}.
\end{equation}

We denote by $(\mathrm{k}/\mathrm{k}_0)^{*}$ the maximal abelian extension of $\mathrm{k}_0$ contained in $\mathrm{k}_3^{(1)}$,
which is called the \textit{relative genus field} of $\mathrm{k}/\mathrm{k}_0$
(cf. \cite[\S\ 2, p. VII-3]{HERZ}).
So the commutator subgroup of $\operatorname{Gal}\left(\mathrm{k}_3^{(1)}/\mathrm{k}_0 \right)$
coincides with $\operatorname{Gal}\left(\mathrm{k}_3^{(1)}/ (\mathrm{k}/\mathrm{k}_0)^{*} \right)$ and thus
$$\operatorname{Gal}\left((\mathrm{k}/\mathrm{k}_0)^{*}/\mathrm{k}_0 \right) \cong
\operatorname{Gal}\left(\mathrm{k}_3^{(1)}/\mathrm{k}_0 \right)/ \operatorname{Gal}\left(\mathrm{k}_3^{(1)}/ (\mathrm{k}/\mathrm{k}_0)^{*} \right).$$

The fact that $\mathrm{k}/\mathrm{k}_0$ is abelian and that $\mathrm{k} \subseteq (\mathrm{k}/\mathrm{k}_0)^{*}$ implies that
$\operatorname{Gal}\left(\mathrm{k}_3^{(1)}/ (\mathrm{k}/\mathrm{k}_0)^{*} \right)$ coincides with $C_{\mathrm{k},3}^{1-\sigma}$,
with the aid of the isomorphism (\ref{IsomCk}) above and by Artin's reciprocity law.
$C_{\mathrm{k},3}^{1-\sigma}$ is called the \textit{principal genus} of $ C_{\mathrm{k},3}$.
Thus
$$\operatorname{Gal}\left( (\mathrm{k}/\mathrm{k}_0)^{*}/ \mathrm{k}  \right) \cong C_{\mathrm{k},3}/C_{\mathrm{k},3}^{1-\sigma}.$$

Since  $\vert C_{\mathrm{k},3}/C_{\mathrm{k},3}^{1-\sigma}\vert= \vert C_{\mathrm{k},3}^{(\sigma)}\vert=3$,
then $(\mathrm{k}/\mathrm{k}_0)^{*}$ is an unramified cyclic cubic extension of $\mathrm{k}$ of degree $3$ which is an abelian extension of $\mathrm{k}_0$. \\
Since $p \equiv 1~(\bmod~3)$, then according to \cite[\S\ 9, Section 1, prop. 9.1.4, p.110]{Clas} we have $p=\pi_1\pi_2$,
where $\pi_1$ and $\pi_2$ are two primes in $\mathrm{k}_{0}$ such that $\pi_2=\pi_1^{\tau}$.
Let $\zeta_{p}$ be a primitive $p$th root of unity.
If we denote by $\mathrm{k}_{p}$ the unique sub-field of $\mathbb{Q}(\zeta_{p})$ of degree $3$ in which only $p$ ramifies,
then $\mathrm{k}_{p}\mathrm{k}_0 = \mathbb{Q}( \zeta_3, \sqrt[3]{\pi_1\pi_2^2})$.
Hence $(\mathrm{k}/\mathrm{k}_0)^{*}=\mathrm{k}(\sqrt[3]{\pi_1\pi_2^2})$.
From the congruence:
$$\left(\pi_1\pi_2^{2}\right)^{\tau}=\pi_2\pi_1^{2}\equiv
\left(\pi_1\pi_2^{2}\right)^{-1}\bmod
\left(\mathrm{k}^{\times}\right)^{3},$$
we conclude by Kummer theory and according to \cite[Proposition 2.4, p. 54]{GERTH2} that:

$$
\vert\left(C_{\mathrm{k},3}/C_{\mathrm{k},3}^{1-\sigma}\right)^{+}\vert=3~~
\emph{and} \
\vert\left(C_{\mathrm{k},3}/C_{\mathrm{k},3}^{1-\sigma}\right)^{-}\vert=1.$$
However, by the observations on the surjective maps $\varphi_{i}$ above, for each integer $i$ such that $0 \leq i \leq s-1 $ we have:

\[  \vert\left(C_{\mathrm{k},3}^{(1-\sigma)^{i}}/
C_{\mathrm{k},3}^{(1-\sigma)^{i+1}}\right)^{+}\vert = \left\{
  \begin{array}{l l}
   3, \ \ \  \text{if}\  s \text{ \ is even,} \\
      1, \ \ \  \text{if}\  s \text{ \ is odd}.    \\
\end{array} \right.\]
and
\[  \vert\left(C_{\mathrm{k},3}^{(1-\sigma)^{i}}/
C_{\mathrm{k},3}^{(1-\sigma)^{i+1}}\right)^{-}\vert = \left\{
  \begin{array}{l l}
   1, \ \ \  \text{if}\  s \text{ \ is even,} \\
       3, \ \ \  \text{if}\  s \text{ \ is odd}.    \\
\end{array} \right.\]

We conclude that:

\[  \vert C_{\mathrm{k},3}^{+}\vert = \left\{
  \begin{array}{l l}
   3^{\frac{s}{2}}, \ \ \  \text{if}\  s \text{ \ is even,} \\
       3^{\frac{s+1}{2}}, \ \ \  \text{if}\  s \text{ \ is odd}.    \\
  \end{array} \right.\]
  and
  \[  \vert C_{\mathrm{k},3}^{-}\vert = \left\{
  \begin{array}{l l}
   3^{\frac{s}{2}}, \ \ \  \text{if}\  s \text{ \ is even,} \\
       3^{\frac{s-1}{2}}, \ \ \  \text{if}\  s \text{ \ is odd}.    \\
  \end{array} \right.\]

Since $\rank  C_{\mathrm{k},3} \in \{1,2\}$, then
$C_{\mathrm{k},3}^{+}$ and $C_{\mathrm{k},3}^{-}$ are a cyclic $3$-groups.
Hence:

$$
  C_{\mathrm{k},3} \simeq  \left\{
    \begin{array}{@{} l c @{}}
      \left(\mathbb{Z}/3^{\frac{s}{2}}\mathbb{Z}\right)^{2}, & \text{if~~} s \,\, \text{is even}, \\
       \mathbb{Z}/3^{\frac{s+1}{2}}\mathbb{Z} \times \mathbb{Z}/3^{\frac{s-1}{2}}\mathbb{Z}, & \text{if~~}
       s \,\, \text{is odd}.
    \end{array}\right.
$$

\end{proof}

\begin{lemma}\label{22}
Let $\mathrm{k}=\mathbb{Q}(\sqrt[3]{p}, \zeta_3)$, where $p$ is a prime number such that $p\equiv 1~(\bmod~3)$.
Let $C_{\mathrm{k},3}^{(\sigma)} = \lbrace \mathcal{A} \in C_{\mathrm{k},3} / ~\mathcal{A}^{\sigma} = \mathcal{A} \rbrace$
be the ambiguous ideal class group of $\mathrm{k}/\mathrm{k}_0$,
and $u$ be the index of subfield units defined as above.
\begin{itemize}
\item[(i)] If $\vert \left(C_{\mathrm{k},3}^{(\sigma)}
  \right)^{+}\vert =3$, then $u=1$,

\item[(ii)] If $\vert \left(C_{\mathrm{k},3}^{(\sigma)}
  \right)^{+}\vert =1$, then $u=3$,
\end{itemize}
where $\left(C_{\mathrm{k},3}^{(\sigma)}
  \right)^{+}$ and $\left(C_{\mathrm{k},3}^{(\sigma)}
  \right)^{-}$ are defined in  \cite[\S\ 2, Lemma 2.1, p. 53]{GERTH2}.
\end{lemma}

\begin{proof}
Let $C_{\mathrm{k},3}$ be the $3$-component of
the class group of $\mathrm{k}$.
According to  \cite[\S\ 2, Lemma 2.1 and Lemma 2.2, p. 53]{GERTH2},
$ C_{\mathrm{k},3} \cong C_{L,3} \times C_{\mathrm{k},3}^{-}$,
where $C_{L,3}$ is the $3$-component of the class group of $L=\mathbb{Q}(\sqrt[3]{p})$.
Let $s$ be the positive integer defined in Lemma \ref{32}.
From proof of Lemma \ref{32} we obtain:
\[  \vert C_{L,3}\vert = \left\{
  \begin{array}{l l}
   3^{\frac{s}{2}}, \ \ \  \text{if}\  s \text{ \ is even,} \\
       3^{\frac{s+1}{2}}, \ \ \  \text{if}\  s \text{ \ is odd,}    \\
  \end{array} \right.\]
  and
  \[  \vert C_{\mathrm{k},3}^{-}\vert = \left\{
  \begin{array}{l l}
   3^{\frac{s}{2}}, \ \ \  \text{if}\  s \text{ \ is even,} \\
       3^{\frac{s-1}{2}}, \ \ \  \text{if}\  s \text{ \ is odd}.    \\
  \end{array} \right.\]

According to \cite[\S\ 14, Theorem 14.1, p. 232]{B-C}, $ \vert C_{\mathrm{k},3}\vert=\dfrac{u}{3}~\vert C_{L,3}\vert^2$, where $u=1$ or $3$.
By Lemma \ref{32}, we have $ \vert C_{\mathrm{k},3}\vert=3^{s}$, thus:
\begin{itemize}
\item[(i)] If $\vert \left(C_{\mathrm{k},3}^{(\sigma)} \right)^{+}\vert =3$ and $\vert \left(C_{\mathrm{k},3}^{(\sigma)} \right)^{-}\vert= 1$,
the integer $s$ is odd, whence $ \vert C_{\mathrm{k},3}\vert$ is not a square, and thus $u=1$.

\item[(ii)]
Conversely, if $\vert \left(C_{\mathrm{k},3}^{(\sigma)} \right)^{+}\vert =1$ and $\vert \left(C_{\mathrm{k},3}^{(\sigma)} \right)^{-}\vert= 3$,
the integer $s$ is even, whence $ \vert C_{\mathrm{k},3}\vert$ is a square, and thus $u=3$.
  \end{itemize}
\end{proof}

\begin{lemma}
\label{NewLemma}
Let $L=\mathbb{Q}(\sqrt[3]{p})$, where $p$ is a prime number such that $p\equiv 1~(\bmod~3)$,
and $\mathrm{k}=\mathbb{Q}(\sqrt[3]{p}, \zeta_3)$.
The prime $p$ decomposes
in the field $\mathrm{k}$ as $\mathcal{P}^{3}\mathcal{Q}^{3}$,
 where $\mathcal{P}$ and $\mathcal{Q}$ are two prime ideals of $\mathrm{k}$,
 and $\mathcal{Q}$ is the conjugate of $\mathcal{P}$. Let $u$ and $s$ be the integers defined as above.
Then the following statements are equivalent:
\begin{itemize}
\item[1)] $u=3 $;
\item[2)]  $s$ is even;
\item[3)]  $\mathcal{P}$ is not principal.
  \end{itemize}
\end{lemma}

\begin{proof} Let  $C_{\mathrm{k},3}^{(\sigma)}$ be the group of ambiguous ideal classes of $\mathrm{k}/\mathrm{k}_0$.
Since $p \equiv 1 (\bmod \ 3)$, then by Lemma \ref{11}, $\vert C_{k,3}^{(\sigma)}\vert=3$.
If we denote by $I_{\mathrm{k}}$ the group of fractional ideals of $\mathrm{k}$, we let 
$S_{\mathrm{k},3}^{(\sigma)}=\lbrace \mathcal{A} \in C_{\mathrm{k},3} ~ | ~ \exists \mathcal{B} \in I_{\mathrm{k}} \text{ such that } \mathcal{A}=[\mathcal{B}] \text{ and } \mathcal{B}^{1-\sigma}=(1)  \rbrace$
be the group of strongly ambiguous ideal classes of the cyclic extension $\mathrm{k} / \mathrm{k}_0$.
It is known that $S_{\mathrm{k},3}^{(\sigma)}$ is generated by the ideal classes of the primes ramified in $\mathrm{k} / \mathrm{k}_{0}$.

\begin{itemize}
\item[$1)\Rightarrow 2): $]  
Let $h_{L,3}$ (resp. $h_{\mathrm{k},3}$)
be the $3$-class number of $L$ (resp $\mathrm{k}$).
By \cite[\S\ 14, Theorem 14.1, p. 232]{B-C}
we have $h_{\mathrm{k},3}= \dfrac{u}{3}h_{L,3}^2$.
If $u=3$, then $ h_{\mathrm{k},3}$ is a square, so the integer $s$ is even.

\item[$2)\Rightarrow 3): $]  
First, since  $p\equiv 1~(\bmod~3)$, then by \cite[\S\ 9, Section 1, prop. 9.1.4, p.110]{Clas} we have
$p=\pi_1\pi_2$, where $\pi_1$ and $\pi_2$ are two primes of $\mathrm{k}_0$ such that
$\pi_1=\pi_2^{\tau}$ and $\pi_2=\pi_1^{\tau}$ and $\pi_1 \equiv \pi_2 \equiv 1 ~(\bmod~3\mathcal{O}_{\mathrm{k}_0})$.
Furthermore,
$\pi_1$ and $\pi_2$ are ramified in $\mathrm{k} / \mathrm{k}_{0}$,
so there exist two prime ideals $\mathcal{P}$ and $\mathcal{Q}$ of $\mathrm{k}$
such that $\pi_1\mathcal{O}_{\mathrm{k}}=\mathcal{P}^{3}$ and $\pi_2\mathcal{O}_{\mathrm{k}}=\mathcal{Q}^{3}$.
Then
$p\mathcal{O}_{\mathrm{k}}=\mathcal{P}^{3}\mathcal{Q}^{3}$. \\
Next, let $s$ be even, where $s$ is the non-null positive integer defined in Lemma \ref{32}. 
We shall prove that the prime ideal $\mathcal{P}$ is not principal.
Assume that $\mathcal{P}$ is principal. 
The fact that $p \equiv 1 \ (\bmod \ 3)$ implies that $\vert C_{\mathrm{k},3}^{(\sigma)}\vert= 3 $ according to Lemma \ref{11}. 
We study all cases in dependence on the congruence class of $p$ modulo $9$:
\begin{itemize}
\item
If $p \equiv 1 \ (\bmod \ 9)$, then $3$ is decomposes in $L$ by \cite[\S\ 4, pp. 51-55]{DED}.
The prime ideals which ramify in $\mathrm{k}/ \mathrm{k}_0$ are $(\pi_{1})$ and $(\pi_{2})$.
Since $\mathcal{P}$ is principal, the prime ideal $\mathcal{Q}$ is also principal.
Thus a generator of $ C_{\mathrm{k},3}^{(\sigma)}$ does not contain
the classes of prime ideals lying above the primes ramified in the extension $\mathrm{k} / \mathrm{k}_{0}$.
So, a generator of $ C_{\mathrm{k},3}^{(\sigma)}$ comes from $ C_{\mathrm{k},3}^{+}$
(cf. \cite[ Proposition 2, part (1), p. 679]{Gras}).
This implies that
$$\vert \left(
C_{\mathrm{k},3}^{(\sigma)} \right)^{+}\vert=3 \ \ \text{and} \ \ \vert \left(
C_{\mathrm{k},3}^{(\sigma)} \right)^{-}\vert=1,$$
where $\left(C_{\mathrm{k},3}^{(\sigma)} \right)^{+}$ and $\left(C_{\mathrm{k},3}^{(\sigma)} \right)^{-}$ are defined in  \cite[\S\ 2, Lemma 2.1, p. 53]{GERTH2}.
By Lemma \ref{22} we see that the integer $s$ is odd.
This contradicts the fact that $s$ is even.
Thus $\mathcal{P}$ is not principal.

\item
If $p \equiv 4\ \text{or} \ 7 \ (\bmod \ 9)$, then $3$ is ramified in $L$ by \cite[\S\ 4, pp. 51-55]{DED},
so the prime ideals which ramify in $\mathrm{k}/ \mathrm{k}_0$ are $(1-\zeta_3)$, $(\pi_{1})$ and  $(\pi_{2})$.
Let $I$ be the unique prime ideal above $(1-\zeta_3)$ in $\mathrm{k}$.
The fact that $p \equiv 4\ \text{or} \ 7 \ (\bmod \ 9)$ implies according to \cite[\S\ 3, Lemma 3.1, p. 16]{AMITA} that
$\pi_{i}\not\equiv 1 \ ( \bmod \ (1-\zeta_3)^3 )$ for $i= \lbrace 1, 2\rbrace$.
Thus, according to \cite[\S\ 3, Lemma 3.3, p. 17]{AMITA} we get $\zeta_3$ is not a norm in the extension $\mathrm{k} / \mathrm{k}_{0} $,
so $S_{\mathrm{k},3}^{(\sigma)}=C_{\mathrm{k},3}^{(\sigma)}$.
Hence, $C_{\mathrm{k},3}^{(\sigma)}$ is generated by the ideal classes of the primes ramified in $\mathrm{k} / \mathrm{k}_{0}$.

\begin{itemize}
\item
If $I$ is not principal, then the ideal class $[I]$ of $I$ generates the group $ C_{\mathrm{k},3}^{(\sigma)}$ and
it is not contained in $C_{\mathrm{k},3}^{1-\sigma}$.
We see that $s=1$ which contradicts the fact that $s$ is even.
\item
If $I$ is principal, and since $\mathcal{P}$ and $\mathcal{Q}$ are also principal, then
 $ C_{\mathrm{k},3}^{(\sigma)}$ will be reduced to $\lbrace 1\rbrace$, which contradict the fact that $\vert C_{k,3}^{(\sigma)}\vert=3$.
\end{itemize}
Hence, the prime ideal $\mathcal{P}$ is not principal.
\end{itemize}

\item[$3)\Rightarrow 1): $]   We note that $\mathcal{P}\mathcal{Q}=(\sqrt[3]{p})$ is a principal ideal.
Assume that $\mathcal{P}$ is not principal, then the prime ideal $\mathcal{Q}=\mathcal{P}^{\tau}$ is also not principal.
Then $\mathcal{P}\mathcal{Q}^{2}$ is also not principal
and the ideal class of $\mathcal{P}\mathcal{Q}^{2}$ generates the group $ C_{\mathrm{k},3}^{(\sigma)}$.
In addition, we have
$\left(\mathcal{P}\mathcal{Q}^{2}\right)^{\tau}=\mathcal{Q}\mathcal{P}^{2}$
which implies that $[\mathcal{P}\mathcal{Q}^{2}]^{\tau}=[\mathcal{Q}\mathcal{P}^{2}]=[\mathcal{P}\mathcal{Q}^{2}]^{-1}$.
Thus, $[\mathcal{P}\mathcal{Q}^{2}]\in\left(C_{\mathrm{k},3}^{(\sigma)}\right)^{-} $.
According to  \cite[\S\ 2, Lemma 2.1, p. 53]{GERTH2},
$C_{\mathrm{k},3}^{(\sigma)}\simeq
\left(C_{\mathrm{k},3}^{(\sigma)} \right)^{+} \times
\left(C_{\mathrm{k},3}^{(\sigma)} \right)^{-} $.
Since $p \equiv 1 \ (\bmod \ 3)$, then by Lemma \ref{11}, $ \vert C_{\mathrm{k},3}^{(\sigma)}\vert=3$,
and as $[\mathcal{P}\mathcal{Q}^{2}]\in \left(C_{\mathrm{k},3}^{(\sigma)}\right)^{-}$, we obtain:
$$\vert \left( C_{\mathrm{k},3}^{(\sigma)} \right)^{+}\vert=1 \ \
\text{and} \ \ \vert \left( C_{\mathrm{k},3}^{(\sigma)}
\right)^{-}\vert=3.$$
We see that the integer $s$ is even, so the class number $h_{\mathrm{k},3}$ is a square. Hence $u=3$.

\end{itemize}

\end{proof}

\begin{lemma}
\label{lem:symb and u}
Let $\mathrm{k}=\mathbb{Q}(\sqrt[3]{p},\zeta_3)$,
where $p$ is a prime number such that $p\equiv 4~or~7~(\bmod~9)$.
Let $u$ be the index of subfield units defined as above. Then:
$$    \left(\frac{3}{p} \right)_3 \neq 1 \Leftrightarrow ( 3 \parallel |C_{L,3}| \ \text{and} \ u=3),$$
where $\left(\frac{.}{p} \right)_3$ is the cubic residue symbol.
\end{lemma}
\begin{proof}
 Lemma \ref{lem:symb and u} is due to Ismaili and El Mesaoudi for $e=0$ in assertion $(1)$ of
\cite[Thm. 3.2, p. 104]{IsEM2}.
\end{proof}



\subsection{Final Proof of Conjecture \ref{Conj1}}
\label{ss:Conj1}

Let $L= \mathbb{Q}(\sqrt[3]{p})$ be a pure cubic field,
where $p$ is a prime number such that $p \equiv 1(\bmod~ 3)$,
and let $\mathrm{k}=\mathbb{Q}(\sqrt[3]{p},\zeta_3)$ be its normal closure.
Let $C_{L,3}$ ( resp. $C_{\mathrm{k},3}$) be the $3$-component of the class group of $L$ ( resp $\mathrm{k}$).\\

$1)$ It is clear that $C_{L,3}$ is  a cyclic group when $s=1$.
Assume that $s\geq 2$.
From the  proof of Lemma \ref{32},
$C_{\mathrm{k},3}^{+}$ and $C_{\mathrm{k},3}^{-}$ are two non-trivial groups,
and according to Lemma \ref{32} we deduce that
$C_{\mathrm{k},3}^{+}$ is a cyclic group,
because otherwise the rank of $C_{\mathrm{k},3}$ will be greater than or equal to $3$.
According to  \cite[\S\ 2, Lemma 2.2, p. 53]{GERTH2}, we conclude that $C_{L,3}$ is a cyclic group.\\

Now, let us show that $p \equiv 1(\bmod~ 9)$
when $C_{L,3}$ contains a cyclic subgroup of order $9$.
This equivalent to the statement that if $p \equiv 4 \text{ or } 7(\bmod~ 9)$ then the class number of $L$ is divisible exactly by $3$. \\
Assume that $p \equiv 4 \text{ or } 7(\bmod~ 9)$ and $9$ divides the class number of $L$,
then $\left(\frac{3}{p} \right)_3=1$.
In fact, if $\left(\frac{3}{p} \right)_3 \neq 1$,
then from the proof of Lemma \ref{lem:symb and u} the class number of $L$ is divisible exactly by $3$ which is absurd. \\
Since  $p=\pi_1\pi_2$, where $\pi_1$ and $\pi_2$ are two primes of $\mathrm{k}_{0}$ such that
$\pi_2=\pi_1^{\tau}$,
then the equations $3 \equiv X^3 \pmod {\pi_1}$ and $ 3 \equiv X^3 \pmod {\pi_2}$ are solvable in $\mathcal{O}_{\mathrm{k}_{0}}$ (cf. \cite[\S\ 9, Section 3, prop. 9.3.3, p.112]{Clas}),
so $3$ is a cubic residue modulo $\pi_1$ and $3$ is a cubic residue modulo $\pi_2$.
According to \cite[\S\ 9, Section 1, prop. 9.1.4, p.110]{Clas}, we have $3=-\zeta_3^2\lambda^2$, where $\lambda=1-\zeta_3$ is a prime of $\mathrm{k}_0$, thus:

$$ \left(\frac{-\zeta_3^2\lambda^2}{\pi_1} \right)_3=\left(\frac{-\zeta_3^2\lambda^2}{\pi_2} \right)_3=1. $$

Since $-1$ is a norm in $ \mathrm{k}_{0}(\sqrt[3]{\pi_1})/\mathrm{k}_{0}$ and a norm in $ \mathrm{k}_{0}(\sqrt[3]{\pi_2})/\mathrm{k}_{0}$, then

      $$ \left(\frac{-1}{\pi_1} \right)_3=\left(\frac{-1}{\pi_2} \right)_3=1, $$
      so
      $$ \left(\frac{\zeta_3^2\lambda^2}{\pi_1} \right)_3=\left(\frac{\zeta_3^2\lambda^2}{\pi_2} \right)_3=1. $$

Hence

           $$ \left(\frac{\zeta_3\lambda}{\pi_1} \right)_3=\left(\frac{\zeta_3\lambda}{\pi_2} \right)_3=1.$$

The fields $\mathrm{k}_{0}(\sqrt[3]{\lambda}) $ and $ \mathrm{k}_{0}(\sqrt[3]{\zeta_3})$ are different. Put $F_1=\mathrm{k}_{0}(\sqrt[3]{\zeta_3})$ and $F_2=\mathrm{k}_{0}(\sqrt[3]{\lambda})$. We have $F_1\neq F_2$, and $ F_1 \cap F_2= \mathrm{k}_{0}$. Let
   $F=  F_1 F_2= \mathrm{k}_{0}(\sqrt[3]{\zeta_3}, \sqrt[3]{\lambda})$.
Then in the subfield diagram
    
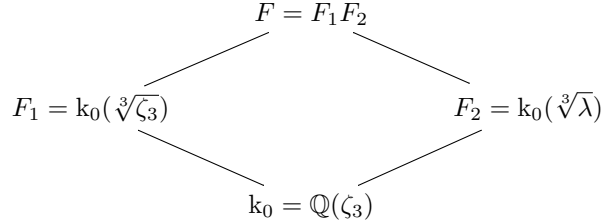
\[
\xymatrix@R=0,7cm{
 {}&{}&{}\\
       {}&{}& F=  F_1 F_2 \ar@{-}[rd] \ar@{-}[ld]&{}&{}\\
       {}& F_1=\mathrm{k}_{0}(\sqrt[3]{\zeta_3}) \ar@{-}[rd]&{}&  F_2=\mathrm{k}_{0}(\sqrt[3]{\lambda}) \ar@{-}[ld] &  &   \\
       {}& {}&\mathrm{k}_{0}=\mathbb{Q}(\zeta_3)& {}& {}\\
  }
\]

we have
$\operatorname{Gal}\left( F_1/\mathrm{k}_{0} \right) $ and $\operatorname{Gal}\left( F_2/\mathrm{k}_{0} \right) $ are cyclic of order $3$,
and $\operatorname{Gal}\left( F/\mathrm{k}_0 \right) $ is abelian.
Since $\pi_1$ is not ramified in $F_1$ and $F_2$, then $\pi_1$ is also not ramified in $F$.
We will calculate the Artin symbol $\left(\dfrac{F/\mathrm{k}_0}{\pi_1} \right)$.\\
On the one hand:
       \begin{eqnarray*}
       \left(\dfrac{F/\mathrm{k}_0}{\pi_1} \right) & = & \left(\dfrac{F_1F_2/\mathrm{k}_0}{(\pi_1)} \right) \\
         & =& \left(\dfrac{F_1/\mathrm{k}_0}{(\pi_1)} \right)\left(\dfrac{F_2/\mathrm{k}_0}{(\pi_1)} \right) \\
         & =& \left(\dfrac{\mathrm{k}_{0}(\sqrt[3]{\zeta_3})/\mathrm{k}_0}{(\pi_1)} \right)\left(\dfrac{\mathrm{k}_{0}(\sqrt[3]{\lambda})/\mathrm{k}_0}{(\pi_1)} \right) \\
         & =& \left(\dfrac{\zeta_3}{\pi_1} \right)_3\left(\dfrac{\lambda}{\pi_1} \right)_3 \\
       & =& \left(\dfrac{\zeta_3 \lambda}{\pi_1} \right)_3 \\
       & = & 1.
       \end{eqnarray*}
On the other hand, since $\mathrm{k}_{0}(\sqrt[3]{\zeta_3})=\mathrm{k}_{0}(\sqrt[3]{\zeta_3^2})$, we get:
     \begin{eqnarray*}
       \left(\dfrac{F/\mathrm{k}_0}{\pi_1} \right) & = & \left(\dfrac{F_1F_2/\mathrm{k}_0}{(\pi_1)} \right) \\
         & = & \left(\dfrac{F_1/\mathrm{k}_0}{(\pi_1)} \right)\left(\dfrac{F_2/\mathrm{k}_0}{(\pi_1)} \right) \\
         & = & \left(\dfrac{\mathrm{k}_{0}(\sqrt[3]{\zeta_3^2})/\mathrm{k}_0}{(\pi_1)} \right)\left(\dfrac{\mathrm{k}_{0}(\sqrt[3]{\lambda})/\mathrm{k}_0}{(\pi_1)} \right) \\
         & = & \left(\dfrac{\zeta_3^2}{\pi_1} \right)_3\left(\dfrac{\lambda}{\pi_1} \right)_3 \\
       & = & \left(\dfrac{\zeta_3^2 \lambda}{\pi_1} \right)_3
       \end{eqnarray*}
We see that  $\left(\dfrac{\zeta_3^2 \lambda}{\pi_1} \right)_3 =1$,
and since $ \left(\dfrac{\zeta_3 \lambda}{\pi_1} \right)_3=1 $, we conclude that
       \begin{equation} \label{pi1}
        \left(\dfrac{\zeta_3}{\pi_1} \right)_3=1,
       \end{equation}
However, since $\pi_2$ is not ramified in $F_1$ and $F_2$,
then $\pi_2$ is also not ramified in $F$.
As above, we calculate the Artin symbol $\left(\dfrac{F/\mathrm{k}_0}{\pi_2} \right)$,
we obtain $\left(\dfrac{\zeta_3^2 \lambda}{\pi_2} \right)_3 =1$,
and since $ \left(\dfrac{\zeta_3 \lambda}{\pi_2} \right)_3=1 $, we get
       \begin{equation} \label{pi2}
        \left(\dfrac{\zeta_3}{\pi_2} \right)_3=1,
       \end{equation}
From the cubic symbols (\ref{pi1}) and (\ref{pi2}) we conclude according to Lemma \ref{14s} that $ \left(\dfrac{\zeta_3}{p} \right)_3=1$, so by Lemma \ref{14ss} we obtain  $p \equiv 1 \pmod 9$
which contradicts the hypothesis that $p \equiv 4 \ \text{or} \ 7  \pmod 9$.
Hence, if $p \equiv 4 ~or ~7(\bmod~ 9)$ then the class number of $L$ is divisible exactly by $3$.\\

$2)$
Let $p$ be a prime number such that $p \equiv 4$ or $7(\bmod~ 9)$.
\begin{itemize}
\item If $\left(\frac{3}{p} \right)_3 \neq 1$, then according to Lemma \ref{lem:symb and u} we get $|C_{L,3}| = 3$ and  $u=1$,
and by Lemma \ref{33} we deduce that
$C_{\mathrm{k},3}\simeq \mathbb{Z}/3\mathbb{Z}$.\\
\item If $\left(\frac{3}{p} \right)_3 = 1$,
then by Lemma \ref{lem:symb and u} we have either $3$ is not divide $|C_{L,3}|$ or $u=3$. Assume that $3$ is not divide $|C_{L,3}|$, then according to \cite[\S 1, Thm, p. 8]{HONDA} the prime $p$ is congruent to   
$-1 \pmod 3$, which is a contradiction because $p \equiv 4 ~or ~7(\bmod~ 9)$. Hence, we have necessary $u=3$ and $3$ divide $|C_{L,3}|$, then $|C_{L,3}| = 3$ from assertion $1)$ above.
By Lemma \ref{33} we deduce that $C_{\mathrm{k},3}\simeq \mathbb{Z}/3\mathbb{Z}\times \mathbb{Z}/3\mathbb{Z} $.
\end{itemize}

\paragraph{}
$3)$
According to Lemma \ref{32}, we get
$\rank  C_{\mathrm{k},3}\in \{1, 2\}$.
On the one hand, for the prime numbers $p=271$ and $p=307$, we have $\left(\frac{3}{p} \right)_3=1$ and $p \equiv 1 (\bmod \ 9)$,
moreover for $p= 271, \ \rank  C_{\mathrm{k},3}=2$ and for $p=307, \ \rank  C_{\mathrm{k},3}=1$.\\
On the other hand, the primes $p=379$ and $p=487$ satisfy $\left(\frac{3}{p} \right)_3\neq 1$ and $p \equiv 1 (\bmod \ 9)$,
moreover for $p= 379, \ \rank  C_{\mathrm{k},3}=1$ and for $p=487, \ \rank  C_{\mathrm{k},3}=2$.\\
This shows that, if $p\equiv 1 (\bmod~ 9)$,
then $\rank C_{\mathrm{k},3} \in \{ 1,2\}$, independently of the value of $\left(\frac{3}{p} \right)_3$.

\paragraph{}
From Lemmas \ref{NewLemma} and \ref{lem:symb and u}, we propose the following Corollary and Proposition:

\begin{corollary}
\label{Corollary}
Let $\mathrm{k}=\mathbb{Q}(\sqrt[3]{p},\zeta_3)$,
where $p$ is a prime number such that $p\equiv 4~or~7~(\bmod~9)$.
Let $u$ be the index of subfield units defined as above. Then:
$$    \left(\frac{3}{p} \right)_3 \neq 1 \Leftrightarrow u=3,$$
where $\left(\frac{.}{p} \right)_3$ is the cubic residue symbol.
\end{corollary}

\begin{proposition}
\label{Conj22}
Let $\mathrm{k}=\mathbb{Q}(\sqrt[3]{p},\zeta_3)$,
where $p$ is a prime number such that $p\equiv 4~or~7~(\bmod~9)$.
Let $I$, $\mathcal{P}$ and $\mathcal{Q}$ be the prime ideals 
defined in the proof of Lemma \ref{NewLemma}. Then:
\begin{itemize}
\item[(i)] The unique prime ideal $\mathcal{P}_0$ above $p$ in $\mathbb{Q}(\sqrt[3]{p})$ is principal independently of the value of $\left(\frac{3}{p}\right)_3$. 
\item[(ii)] $\left(\frac{3}{p} \right)_3 = 1$ if and only if $I$ is principal and $\mathcal{P}$ (resp. $\mathcal{Q}$) is not principal.
\item[(iii)] $\left(\frac{3}{p} \right)_3 \neq 1$ if and only if $I$ is not principal and $\mathcal{P}$ (resp. $\mathcal{Q}$) is principal.
\end{itemize}
where $\left(\frac{.}{p} \right)_3$ is the cubic residue symbol.
\end{proposition}

The Proposition \ref{Conj22} will be underpinned by numerical examples
obtained with the computational number theory system PARI \cite{PARI}
in \S\ \ref{Appendix}.

\begin{proof}
Since $p\equiv 1 \ (\bmod \ 3$), then according to \cite[\S\ 9, Section 1, prop. 9.1.4, p.110]{Clas}, $p=\pi_1\pi_2$,
where $\pi_1$ and $\pi_2$ are two primes of $\mathrm{k}_{0}$ such that
$\pi_2=\pi_1^{\tau}$ and $\pi_1 \equiv \pi_2 \equiv 1 ~(\bmod~3\mathcal{O}_{\mathrm{k}_0})$.
As $p \equiv 4 \ \text{or} \ 7 (\bmod \ 9)$, then $3$ is ramified in $L$ by \cite[\S\ 4, pp. 51-55]{DED},
so the primes ramified in $\mathrm{k} / \mathrm{k}_{0}$ are $(1-\zeta_3)$, $\pi_1$ and $\pi_2$.
Put $(\pi_1)= \mathcal{P}^{3}, \ (\pi_2)=\mathcal{Q}^{3}$ and $(1-\zeta_3)=I^{3}$. \\
The fact that $p \equiv 4\ \text{or} \ 7 \ (\bmod \ 9)$ implies that $S_{\mathrm{k},3}^{(\sigma)}=C_{\mathrm{k},3}^{(\sigma)}$, where $S_{\mathrm{k},3}^{(\sigma)}$ and $C_{\mathrm{k},3}^{(\sigma)}$ are defined in the proof of Lemma  \ref{NewLemma}.
Then $C_{\mathrm{k},3}^{(\sigma)}$ is generated by the ideal classes of the primes ramified in $\mathrm{k} / \mathrm{k}_{0}$. 

$\\ (i)$ Let $\mathcal{P}_0$ be the unique prime ideal above $p$ in $\mathbb{Q}(\sqrt[3]{p})$, we have $p\mathcal{O}_{L}=\mathcal{P}_{0}^{3}$, and since $p\mathcal{O}_{k}=\mathcal{P}^{3}\mathcal{Q}^{3}$, then the ideal $\mathcal{P}_{0}=\mathcal{P}\mathcal{Q}=(\sqrt[3]{p})$ is principal.

$\\ (ii)$
Assume that  $\left(\frac{3}{p} \right)_3 = 1$, then according to Corollary \ref{Corollary} we have $u=3$, and by Lemma \ref{NewLemma} the prime ideal $\mathcal{P}$ is not principal, so  $\mathcal{Q}=\mathcal{P}^{\tau}$ is also not principal. Since $\vert C_{k,3}^{(\sigma)}\vert=3$ by Lemma \ref{11},  then $I$ is principal.

$\\ (iii)$ Reasoning as in $(ii)$. Assume that  $\left(\frac{3}{p} \right)_3 \neq 1$, then according to Corollary \ref{Corollary} we have $u=3$, and by Lemma \ref{NewLemma} the ideal $\mathcal{P}$ is principal, so  $\mathcal{Q}=\mathcal{P}^{\tau}$ is also principal. Since $\vert C_{k,3}^{(\sigma)}\vert=3$ by Lemma \ref{11},  then $I$ is not principal.

\end{proof}

\begin{remark}
Let $L=\mathbb{Q}(\sqrt[3]{d})$, where $d>1$ is a cube-free integer,
let $\mathrm{k}=\mathbb{Q}(\sqrt[3]{d},\zeta_3)$ be the normal closure of the pure cubic field $L$
and $C_{L,3}$ (resp. $C_{\mathrm{k},3}$) be the $3$-component of the class group of $ L$ (resp. $\mathrm{k}$).
$\\ 1)~$ If $|C_{L,3}| = 3$, then $\rank   C_{\mathrm{k},3} \leq 2$.
$\\ 2)~$ If $|C_{L,3}| = 9$, then $\rank
C_{\mathrm{k},3} \leq 3$ if $u=1$, and $\rank  C_{\mathrm{k},3} \leq
4$ otherwise.
$\\ 3)~$ If $d=p$ or $p^2$, with $p$ is a prime number such that $p \equiv 1 \pmod 9$, and if $9$ divide exactly the class number of $\mathbb{Q}(\sqrt[3]{p})$ and $u=1$, then according to \cite[\S\ 1, Theorem 1.1, p. 1]{CaractArXiv}, the $3$-class group of $\mathbb{Q}(\sqrt[3]{p},\zeta_3)$ is of type $(9,3).$ Furthermore, if $3$ is not residue cubic modulo $p$, then a generators of $3$-class group of $\mathbb{Q}(\sqrt[3]{p},\zeta_3)$ can be deduced by \cite[\S\ 3, Theorem 3.2, p. 10]{GENArXiv}.
\end{remark}

\section{Appendix}\label{Appendix}
\subsection{Illustrations of Conjecture \ref{Conj1}}
\label{NumExp}
Let $p\equiv 1\,(\mathrm{mod}\,3)$ be a prime, $L=\mathbb{Q}(\sqrt[3]{p})$, and $\mathrm{k}=\mathbb{Q}(\sqrt[3]{p},\zeta_3)$.  
Let $u$ be the index of subfield units defined as above, $C_{L,3}$ (resp. $C_{\mathrm{k},3}$) be the $3$-class group of $L$ (resp. $\mathrm{k}$). 
\begin{center}
Table 1: Some numerical examples for  Conjecture  \ref{Conj1}.

\end{center}
\begin{longtable}{| c | c | c | c | c | c | c | c | c | c | c | c | c |  }
\hline
    $p$ & $p ~(\bmod~ 9)$  & $u$ & $(\frac{3}{p})_3$ & $C_{L,3}$ & $ \rank  C_{\mathrm{k},3}$ \\
\hline
\endfirsthead
\hline
    $p$  & $p ~(\bmod~ 9)$  & $u$ & $(\frac{3}{p})_3$ & $C_{L,3}$ & $ \rank  C_{\mathrm{k},3}$ \\
\hline
\endhead
   $199$ & $1$  & $1$ & & $[9]$ & $2$   \\

   $211$ & $4$  & $1$ & $\neq 1$ & $[3]$ & $1$   \\

   $223$  & $7$  & $1$ & $\neq 1$ & $[3]$ & $1$   \\

   $367$  & $7$  & $3$ & $1$ & $[3]$ & $2$   \\

   $499$  & $4$  & $3$ & $1$ & $[3]$ & $2$   \\

   $541$  & $1$  & $3$ & & $[9]$ & $2$   \\

\hline

\end{longtable}
Moroever, in  Section $17$ of \cite[Numerical Data, p. 238]{B-C}, and also in the tables of \cite{TABdata} which give  the class number of a pure cubic field, the prime numbers $p=61, 67, 103,$ and $151$, which are all congruous to $4 $ or $ 7 \ \pmod 9$, verify the following  properties:
\begin{itemize}
\item[$i)$] $3$ is a residue cubic modulo $p$;
\item[$ii)$] $3$ divide exactly the class number of  $L$;
\item[$iii)$]  $u=3$;
\item[$iv)$] $C_{L,3} \simeq \mathbb{Z}/3 \mathbb{Z}$, and $C_{\mathrm{k},3}\simeq \mathbb{Z}/3 \mathbb{Z}\times \mathbb{Z}/3 \mathbb{Z}.$
\end{itemize}

\subsection{Illustrations of Proposition \ref{Conj22} and Corollary \ref{Corollary} } \label{conj22App}
Let $L=\mathbb{Q}(\sqrt[3]{p})$, and $\mathrm{k}=\mathbb{Q}(\sqrt[3]{p},\zeta_3)$,
where $p$ is a prime such that $p\equiv 4~or~7~(\bmod~9)$.
We put $3 \mathcal{O}_{L}=I_0^3$,  $p\mathcal{O}_{L}=\mathcal{P}_0^3$, $3 \mathcal{O}_{\mathrm{k}}=I^6$, and  $p \mathcal{O}_{\mathrm{k}}=\mathcal{P}^3\mathcal{Q}^3$. \\

\begin{center}
Table 2: Case where  $p \equiv 4 \ or \ 7 \pmod 9$ and $(\frac{3}{p})_3=1$.
\end{center}

\begin{longtable}{| c | c | c | c | c | c | c | c | c |  }
\hline
   $p$ & $u$ & $(\frac{3}{p})_3$  & $C_{L,3}$ & $C_{\mathrm{k},3}$ & $I_0$ & $I$ & $\mathcal{P}_0$ & $\mathcal{P}$      \\
\hline
\endfirsthead
\hline
  $p$ & $u$ & $(\frac{3}{p})_3$  & $C_{L,3}$ & $C_{\mathrm{k},3}$ & $I_0$ & $I$ & $\mathcal{P}_0$ & $\mathcal{P}$  \\
\hline
\endhead
   $61$ & $3$ & $1$ & $[3]$ & $[3, 3]$ & $[0]~$   & $[0]~$  & $[0]~$ & $\neq[0]~$  \\
   $67$ & $3$ & $1$ & $[3]$ & $[3, 3]$ & $[0]~$   & $[0]~$  & $[0]~$ & $\neq[0]~$  \\
   $103$ & $3$ & $1$ & $[3]$ & $[3, 3]$ & $[0]~$   & $[0]~$  & $[0]~$ & $\neq[0]~$  \\
   $151$ & $3$ & $1$ & $[3]$ & $[3, 3]$ & $[0]~$   & $[0]~$  & $[0]~$ & $\neq[0]~$  \\
   $193$ & $3$ & $1$ & $[3]$ & $[3, 3]$ & $[0]~$   & $[0]~$  & $[0]~$ & $\neq[0]~$  \\
   $367$ & $3$ & $1$ & $[3]$ & $[3, 3]$ & $[0]~$   & $[0]~$  & $[0]~$ & $\neq[0]~$  \\
   $439$ & $3$ & $1$ & $[3]$ & $[3, 3]$ & $[0]~$   & $[0]~$  & $[0]~$ & $\neq[0]~$  \\
   $499$ & $3$ & $1$ & $[3]$ & $[3, 3]$ & $[0]~$   & $[0]~$  & $[0]~$ & $\neq[0]~$  \\
   $547$ & $3$ & $1$ & $[3]$ & $[3, 3]$ & $[0]~$   & $[0]~$  & $[0]~$ & $\neq[0]~$  \\
   $619$ & $3$ & $1$ & $[3]$ & $[3, 3]$ & $[0]~$   & $[0]~$  & $[0]~$ & $\neq[0]~$  \\
   $643$ & $3$ & $1$ & $[3]$ & $[3, 3]$ & $[0]~$   & $[0]~$  & $[0]~$ & $\neq[0]~$  \\
   $661$ & $3$ & $1$ & $[3]$ & $[3, 3]$ & $[0]~$   & $[0]~$  & $[0]~$ & $\neq[0]~$  \\
   $727$ & $3$ & $1$ & $[3]$ & $[3, 3]$ & $[0]~$   & $[0]~$  & $[0]~$ & $\neq[0]~$  \\
   $787$ & $3$ & $1$ & $[3]$ & $[3, 3]$ & $[0]~$   & $[0]~$  & $[0]~$ & $\neq[0]~$  \\
   $853$ & $3$ & $1$ & $[3]$ & $[3, 3]$ & $[0]~$   & $[0]~$  & $[0]~$ & $\neq[0]~$  \\
   $967$ & $3$ & $1$ & $[3]$ & $[3, 3]$ & $[0]~$   & $[0]~$  & $[0]~$ & $\neq[0]~$  \\
   $997$ & $3$ & $1$ & $[3]$ & $[3, 3]$ & $[0]~$   & $[0]~$  & $[0]~$ & $\neq[0]~$  \\
\hline
\end{longtable}

\begin{center}
Table 3: Case where $p \equiv 4 \ or \ 7 \pmod 9$ and $(\frac{3}{p})_3\neq 1$.
\end{center}
\begin{longtable}{| c | c | c | c | c | c | c | c | c |  }
\hline
   $p$ & $u$ & $(\frac{3}{p})_3$ & $C_{L,3}$ & $C_{\mathrm{k},3}$ & $I_0$ & $I$ & $\mathcal{P}_0$ & $\mathcal{P}$      \\
\hline
\endfirsthead
\hline
  $p$ & $u$ & $(\frac{3}{p})_3$  & $C_{L,3}$ & $C_{\mathrm{k},3}$ & $I_0$ & $I$ & $\mathcal{P}_0$ & $\mathcal{P}$  \\
\hline
\endhead
   $7$ & $1$ & $\neq 1$ & $[3]$ & $[3]$ & $\neq[0]~$   & $\neq[0]~$  & $[0]~$ & $[0]~$  \\
   $13$ & $1$ & $\neq 1$ & $[3]$ & $[3]$ & $\neq[0]~$   & $\neq[0]~$  & $[0]~$ & $[0]~$  \\
   $31$ & $1$ & $\neq 1$ & $[3]$ & $[3]$ & $\neq[0]~$   & $\neq[0]~$  & $[0]~$ & $[0]~$  \\
   $43$ & $1$ & $\neq 1$ & $[3]$ & $[3]$ & $\neq[0]~$   & $\neq[0]~$  & $[0]~$ & $[0]~$  \\
   $79$ & $1$ & $\neq 1$ & $[3]$ & $[3]$ & $\neq[0]~$   & $\neq[0]~$  & $[0]~$ & $[0]~$  \\
   $97$ & $1$ & $\neq 1$ & $[3]$ & $[3]$ & $\neq[0]~$   & $\neq[0]~$  & $[0]~$ & $[0]~$  \\
   $139$ & $1$ & $\neq 1$ & $[3]$ & $[3]$ & $\neq[0]~$   & $\neq[0]~$  & $[0]~$ & $[0]~$  \\
   $157$ & $1$ & $\neq 1$ & $[3]$ & $[3]$ & $\neq[0]~$   & $\neq[0]~$  & $[0]~$ & $[0]~$  \\
   $211$ & $1$ & $\neq 1$ & $[3]$ & $[3]$ & $\neq[0]~$   & $\neq[0]~$  & $[0]~$ & $[0]~$  \\
   $223$ & $1$ & $\neq 1$ & $[3]$ & $[3]$ & $\neq[0]~$   & $\neq[0]~$  & $[0]~$ & $[0]~$  \\
   $229$ & $1$ & $\neq 1$ & $[3]$ & $[3]$ & $\neq[0]~$   & $\neq[0]~$  & $[0]~$ & $[0]~$  \\
   $241$ & $1$ & $\neq 1$ & $[3]$ & $[3]$ & $\neq[0]~$   & $\neq[0]~$  & $[0]~$ & $[0]~$  \\
   $277$ & $1$ & $\neq 1$ & $[3]$ & $[3]$ & $\neq[0]~$   & $\neq[0]~$  & $[0]~$ & $[0]~$  \\
   $283$ & $1$ & $\neq 1$ & $[3]$ & $[3]$ & $\neq[0]~$   & $\neq[0]~$  & $[0]~$ & $[0]~$  \\
   $313$ & $1$ & $\neq 1$ & $[3]$ & $[3]$ & $\neq[0]~$   & $\neq[0]~$  & $[0]~$ & $[0]~$  \\
   $331$ & $1$ & $\neq 1$ & $[3]$ & $[3]$ & $\neq[0]~$   & $\neq[0]~$  & $[0]~$ & $[0]~$  \\
\hline
\end{longtable}


\section{Acknowledgements}
\label{s:Thanks}

\noindent
The authors would like to thank Professor Daniel C. Mayer who was of a great help concerning correcting the spelling mistakes and making some precious modifications 
that gave more value and meaning to the work.



\begin{thebibliography}{XX}


%
\bibitem{AMITA}
S. Aouissi, D. C. Mayer, M. C. Ismaili, M. Talbi, and A. Azizi,
\emph{$3$-rank of ambiguous class groups in cubic Kummer extensions},
arXiv preprint, 2018,
\texttt{arXiv:1804.00767v3}.
%
\bibitem{CaractArXiv}
S. Aouissi, M. Talbi, M. C. Ismaili and A. Azizi. \emph{Fields $\mathbb {Q}(\sqrt [3]{d},\zeta_3) $ whose $3 $-class group is of type $(9, 3) $.} arXiv preprint, 2018,
\texttt{arXiv:1805.04963}.
%
%
\bibitem{GENArXiv}
S. Aouissi, M. C. Ismaili, M. Talbi and A. Azizi. \emph{The generators of $3 $-class group of some fields of degree $6 $ over $\mathbb {Q} $.}
 arXiv preprint, 2018,
\texttt{arXiv:1804.00692}.
%
\bibitem{B-C}
P. Barrucand and H. Cohn,
\emph{Remarks on principal factors in a relative cubic field},
J. Number Theory \textbf{3} (1971), No. 2, 226--239.
%
\bibitem{BC2}
P. Barrucand and H. Cohn,
\emph{A rational genus, class number divisibility, and unit theory for pure cubic fields},
J. Number Theory \textbf{2} (1970), No. 1, 7--21.
%
\bibitem{BWB}
P. Barrucand, H. C. Williams and L. Baniuk,
\emph{A computational technique for determining the class number of a pure cubic field},
Math. Comp. \textbf{30} (1976), No. 134, 312--323.
%
\bibitem{TABdata}
B. D. Beach, H. C. Williams, C. R. Zarnke
\emph{Some computer results on units in quadratic and cubic fields}, Proceeding of the twenty-fifth summer meeting of the Canadian Mathematical Congress (Lake Head Univ., Thunder Bay, Out. 1971), 609-648.
%
%
\bibitem{ISL}
R. A. Mollin,
\emph{Algebraic number theory},
London, UK: Chapman and Hall (1999).
%
\bibitem{DED}
R. Dedekind,
\emph{\"{U}ber die Anzahl der Idealklassen in reinen kubischen Zahlk\"{o}rpern},
J. Reine Angew. Math. \textbf{121} (1900), 40--123.
%
\bibitem{GERTH1}
F. Gerth III,
\emph{On $3$-class groups of cyclic cubic extensions of certain number fields},
J. Number Theory \textbf{8} (1976), No. 1, 84--98.
%
\bibitem{GERTH2}
F. Gerth III,
\emph{On $3$-class groups of pure cubic fields},
J. Reine Angew. Math  \textbf{278/279} (1975), 52--62.
%
\bibitem{Ge2005}
F. Gerth III,
\emph{On $3$-class groups of certain pure cubic fields},
Bull. Austral. Math. Soc. \textbf{72} (2005), 471--476.
%
\bibitem{Gras}
G. Gras,
\emph{Sur les $\ell$-classes d'id\'{e}aux des extensions non galoisiennes de $\mathbf{Q}$
de degr\'{e} premier impair $\ell$ a cloture galoisienne di\'{e}drale de degr\'{e} $2\ell$},
J. Math. Soc. Japan \textbf{26} (1974), 677-685.
%
\bibitem{Hass}
H. Hasse, 
\emph{The class number formula of Dedekind-Meyer for simply real cubic fields},
unpublished manuscript.
%
\bibitem{HERZ}
C. S. Herz,
\emph{Construction of Class Fields}, Chapter VII in:
A. Borel, et al.,
Seminar on Complex Multiplication (Inst. Advanced Study, Princeton, 1957--58),
Lecture Notes in Math. \textbf{21},
Springer-Verlag, New York, 1966.
%
\bibitem{HONDA}
T.Honda,
\textit{Pure cubic fields whose class numbers are multiples of three},
J.Number Theory \textbf{3} (1971), No. 1, 7--12.
%
\bibitem{Clas}
K. Ireland and M. Rosen,
\emph{A Classical Introduction to Modern Number Theory},
Graduate Texts in Math. \textbf{84},
Springer-Verlag, New York, 1982.
%
\bibitem{IsEM2}
M. C. Ismaili and R. El Mesaoudi,
\textit{Corps cubiques purs dont le nombre de classes est exactement divisible par $3$},
Ann. Sci Math. Qu\'{e}bec \textbf{28} (2004), No. 1--2, 103--112 (2005).
%
\bibitem{RECLAWS}
F. Lemmermeyer, \emph{Reciprocity laws: from Euler to Eisenstein}. Springer Science and Business Media, 2013.
%
\bibitem{FLEM}
F. Lemmermeyer,
\emph{Class field towers},
Chapter 1, Section 1.10. Unsolved Problems, Conjecture 5, page 44 (2010).
%
\bibitem{Markf}
A. A. Markoff,
\emph{Sur les nombres entiers d\'{e}pendants d'une racine cubique d'un nombre entier ordinaire},
M\'{e}m. Acad. Imp. Sci. St. P\'{e}tersbourg (S\'{e}rie VII) \textbf{38} (1892), No. 9, 1--37.
%
\bibitem{Mey}
C. Meyer,
\emph{Die Berechnung der Klassenzahl Abelscher Zahlk\"{o}rper \"{u}ber quadratischen Zahlk\"{o}rpern},
Akademie Verlag, Berlin, 1957.
%
\bibitem{PARI}
PARI Developer Group, PARI/GP, Version 2.9.4,
Bordeaux, 2017, \\
\texttt{http://pari.math.u-bordeaux.fr}.
%
\bibitem{HCW1982}
H. C. Williams,
\emph{Determination of principal factors in $\mathbb{Q}(\sqrt{D})$ and $\mathbb{Q}(\sqrt[3]{D})$},
Math. Comp. \textbf{38} (1982), No. 157, 261--274.

\end{thebibliography}
\end{document}